\documentclass[reqno]{amsart}
\usepackage{amssymb,mathrsfs, amsmath,amssymb,amsfonts, amsthm, dsfont, tikz}
\usepackage{hyperref}
\newtheorem{theorem}{Theorem}[section]
\newtheorem{lemma}[theorem]{Lemma}
\newtheorem{corollary}[theorem]{Corollary}

\theoremstyle{definition}
\newtheorem{definition}[theorem]{Definition}

\theoremstyle{remark}
\newtheorem{remark}[theorem]{Remark}

\numberwithin{equation}{section}

\begin{document}
\title[Quantum unique ergodicity and the number of nodal domains]{Quantum unique ergodicity and the number of nodal domains of eigenfunctions}
\author{Seung uk Jang}
\address{Department of Mathematics, University of Chicago, Chicago, IL 60637, USA}
\email{seungukj@uchicago.edu}
\author{Junehyuk Jung}
\address{Department of Mathematical Science, KAIST, Daejeon 305-701, South Korea}
\curraddr{School of Mathematics, IAS, Princeton, NJ 08540, USA}
\email{junehyuk@ias.edu}

\thanks{This work was supported by the National Research Foundation of Korea (NRF) grant funded by the Korea government(MSIP)(No. 2013042157) and by the National Science Foundation under agreement No. DMS-1128155. The first author was partially supported by the TJ Park Post-doc Fellowship funded by POSCO TJ Park Foundation. We thank P. Sarnak, S. Zelditch, A. Reznikov, J. Toth, S. Eswarathasan, M. Magee, and Y. Canzani for many helpful discussions and enlightening comments.}
\begin{abstract}
 We prove that the Hecke--Maass eigenforms for a compact arithmetic
triangle group have a growing number of nodal domains as the eigenvalue tends to $+\infty$. More generally the same is proved for eigenfunctions
on negatively curved surfaces that are even or odd with respect to a geodesic symmetry
and for which Quantum Unique Ergodicity holds.
\end{abstract}

\maketitle

\section{Introduction}
\subsection{Nodal domains of eigenfunctions on a surface}
Let $(M,g)$ be a smooth compact Riemannian surface without boundary and let $\{u_n\}$ be an orthonormal Laplacian eigenbasis ordered by the eigenvalue, i.e.,
\begin{align*}
-\Delta_g u_n &= \lambda_n^2 u_n\\
\langle u_n,u_m \rangle_M&=\delta_{nm}\\
0=\lambda_0 &<\lambda_1 \leq \lambda_2 \leq \ldots,
\end{align*}
where $\Delta_g$ is the Laplace--Beltrami operator on $M$. Here $\langle f,h \rangle_M=\int_M f\bar{h}dV_g$, where $dV_g$ is the volume form of the metric $g$. We assume throughout the paper that every eigenfunction is real valued. We denote by $Z_{u_n}$ the nodal set $\{x \in M~:~ u_n(x)=0\}$ of $u_n$ and by $\mathcal{N}(u_n)$ the number of nodal domains of $u_n$, where nodal domains are the connected components of $M \backslash Z_{u_n}$.

The purpose of this paper is to understand the growth of $\mathcal{N}(u_n)$ as $n$ tends to $+\infty$. Note that Courant's nodal domain theorem \cite{ch53} and Weyl law imply that $\mathcal{N}(u_n)= O(\lambda_n^2)$. However it is not true in general that the number of nodal domains necessarily grows with the eigenvalue. For instance, when $M=S^2$ (the standard sphere) or $M=T^2$ (the flat torus), there exists a sequence of eigenfunctions $\{u_{n_k}\}$ with $\lambda_{n_k}\to \infty$ that satisfy $\mathcal{N}(\phi_{n_k})\leq 3$ \cite{st,lewy,crit}.

We first state the main result of the paper.
\begin{theorem}\label{arith}
Let $\phi$ be a Hecke--Maass eigenform for an arithmetic triangle group with eigenvalue $\lambda$. Then we have $\lim_{\lambda \to +\infty} \mathcal{N}(\phi)=+\infty$.
\end{theorem}
Note that there are $76$ arithmetic triangle groups \cite{MR0429744} which are divided into $18$ commensurable classes \cite{MR0463116}.
\begin{remark}
This result in the stronger form of a lower bound of $\gg_{\epsilon} \lambda^{\frac{1}{12}-\epsilon}$ for the number of nodal domains is obtained in \cite{GRS2}, however assuming the Generalized Lindel{\"o}f Hypothesis for a certain family of $L$-functions.
\end{remark}
Theorem \ref{arith} is a consequence of the below Theorem \ref{thm1} which considers the number of nodal domains when we have Quantum Unique Ergodicity(QUE). Note that arithmetic Quantum Unique Ergodicity theorem by Lindenstrauss \cite{lin06} asserts that QUE holds for Maass--Hecke eigenforms on these triangles. In order to state Theorem \ref{thm1}, we first fix $a \mapsto a(x,hD)$, a quantization of a symbol $a(x,\xi) \in C^\infty\left(T^*M\right)$, to a pseudo-differential operator. (We briefly review semi-classical analysis on a manifold in the Appendix, and refer the readers to \cite{zss} for detailed discussion on the subject.) We say QUE holds for the sequence of eigenfunctions $\{u_n\}_{n\geq 1}$ if we have
\begin{equation}\label{qeseq}
\lim_{n \to \infty}\left\langle a\left(x,\lambda_{n}^{-1}D\right) u_{n},u_{n}\right\rangle_M =\int_{S^*M} a(x,\xi) d\mu
\end{equation}
for any fixed symbol $a \in C^\infty \left(T^*M\right)$ of finite order. Here $d\mu$ is a normalized Liouville measure on the unit cotangent bundle $S^*M$. We often write $\mathrm{Op}(a)$ for an operator that acts on an eigenfunction $u$ with the eigenvalue $\lambda$ as $a(x,\lambda^{-1}D)$.
\begin{remark}
The classical notions of equidistribution of these ``Wigner measures'' \cite{snirel,cd22,zeld1} are concerned with \eqref{qeseq} for degree zero homogeneous symbols. We show in \S\ref{finite} that the degree zero homogeneous case implies \eqref{qeseq}. %This allows us to use \eqref{qeseq} in \S\ref{conse} for both any sequence for which QUE holds and the density $1$ sequence in the quantum ergodic case \cite{snirel,cd22,zeld1}.
\end{remark}
\begin{remark}
For a compact smooth negatively curved Riemannian manifold, it is conjectured by Rudnick and Sarnak \cite{rs94} that QUE holds for any given orthonormal eigenbasis $\{u_{n}\}$.
\end{remark}
\begin{theorem}\label{thm1}
Let $M$ be a smooth compact Riemannian surface without boundary. Assume that there exists an orientation-reversing isometric involution $\tau:M \to M$ such that $\mathrm{Fix}(\tau)$ is separating. Let $\{u_n\}$ be an orthonormal basis of $L^2(M)$ such that each $u_n$ is a joint eigenfunction of the Laplacian and $\tau$. Assume that QUE holds for the sequence $\{u_n\}$. Then
\[
\lim_{n\to \infty} \mathcal{N}(u_{n}) = +\infty.
\]
\end{theorem}
We say a function $f$ on $M$ is even (resp. odd) if $\tau f=f$ (resp. $\tau f=-f$). In order to prove Theorem \ref{thm1}, we first use a topological argument to bound the number of nodal domains of an even (resp. odd) eigenfunction from below by the number of sign changes (resp. the number of singular points) of the eigenfunction along $\mathrm{Fix}(\tau)$. Such an argument is first developed in \cite{GRS} and we review in \S\ref{sec1} in terms of the nodal graphs and Euler's inequality as in \cite{JZ1}. We then use Bochner's theorem and a Rellich type identity to deduce from QUE that even (resp. odd) eigenfunctions $\{u_n\}$ have a growing number of sign changes (resp. singular points) along $\mathrm{Fix}(\tau)$ as $n$ tends to $+\infty$. This is the main contribution of the paper, and we sketch the argument in the following section. %\S\ref{Sketch}.
\begin{remark}
In \cite{JZ1}, the same assertion has been obtained when $M$ is a negatively curved surface, but for a density one subsequence of $\{u_n\}$. The argument of \cite{JZ1} to detect a sign change of an eigenfunction $u_n$ on a curve $\beta$ is to compare
\[
\left|\int_\beta u_n(s) ds\right| ~~\text{and}~~\int_\beta |u_n(s)| ds.
\]
(See \cite{GRS,jung3,br,JZ1,mm,GRS2}, where such an idea is used to prove a lower bound for the number of sign changes in various contexts.) In order to bound $\|u_n\|_{L^1(\beta)}$ from below using H{\"o}lder's inequality, the authors use the Quantum Ergodic Restriction (QER) theorem \cite{tz1,dz} for the lower bound of $\|u_n\|_{L^2(\beta)}$ and the point-wise Weyl law with an improved error term \cite{jzBe} for the upper bound of $\|u_n\|_{L^\infty(\beta)}$. For the upper bound of the integral of $u_n$ over $\beta$, the authors use the Kuznecov sum formulae \cite{zcuz}. Note that the result of \cite{jzBe} requires a global assumption on the geometry of $M$ that it does not have conjugate points, which is satisfied if $M$ is negatively curved. Also note that in order to bound such quantities using QER theorem and Kuznecov sum formulae, it is necessary to remove a density $0$ subsequence.
\end{remark}
\subsection{Sketch of the proof: sign changes of even eigenfunctions}\label{Sketch}
The main step in the proof of Theorem \ref{thm1} is to show that all but finitely many $u_{n}$ have at least one sign change on any given fixed segment $\beta$ of $\mathrm{Fix}(\tau)$.

To simplify the discussion, let $\{\psi_n\}$ be a sequence of functions in $C_0^\infty ([0,1])$. Assume that for any fixed integer $m \geq 0$ we have
\begin{equation}\label{mom}
\lim_{n \to \infty} \int_0^1 \left|\frac{\partial^m \psi_n}{\partial s^m}(s)\right|^2 ds = a_{2m},
\end{equation}
for some positive real number $a_{2m}$. Let $h_n(\xi)=|g_n(\xi)|^2/\|g_n\|_2^2$, where $g_n(\xi)$ is the Fourier transform of $\psi_n$,
\[
g_n(\xi) = (2\pi )^{-\frac{1}{2}} \int_0^1 e^{i\xi s} \psi_n(s) ds.
\]
Assume that there exists a unique probability measure $d\mu(\xi)$ whose $2m$th moment is $a_{2m}/a_0$ and whose $(2m+1)$th moment is zero for any $m\geq 0$. Then \eqref{mom} implies that a sequence of probability measures $h_n(\xi) d\xi$ converges to $d\mu(\xi)$ in moments.

We claim that all but finitely many $\psi_n$ have at least one sign change on $(0,1)$ under the assumption that $d\mu(\xi)$ is not positive-definite, i.e., not a Fourier transform of a positive measure(Theorem \ref{prob3}). Assume for contradiction that there exists a subsequence $\{\psi_{n_k}\}$ of $\{\psi_n\}$ such that $\psi_{n_k}$ does not change sign on $(0,1)$ for all $k$. Then by Bochner's theorem, $\{h_{n_k}(\xi)\}$ is a sequence of positive-definite functions and it cannot converge in moments to a measure that is not positive-definite, contradicting the assumption that $d\mu(\xi)$ is not positive-definite.

Now let $f \in C_0^\infty (\beta)$ be a nonnegative function. Our aim is to apply the above argument to $\psi_n(s)=f(s)u_{n}|_{\beta}(s)$, when QUE holds for the sequence of eigenfunctions $\{u_{n}\}$. Note that it is not necessarily true that the limit
\[
\lim_{n \to \infty}\int_{\beta}|\psi_n(s)|^2 ds
\]
should exists. However, under the assumption that QUE holds for $\{u_{n}\}$, we may instead compute the limit (Theorem \ref{awe})
\begin{equation}\label{eqhaha}
\lim_{n \to \infty} \int_\beta \left|\psi_n(s)\right|^2 - \left|\frac{1}{\lambda_{n}^m}\frac{\partial^m \psi_n}{\partial s^m}(s)\right|^2 ds = 2(1- b_{2m}) \int_{\beta} f^2(s) ds
\end{equation}
for each fixed $m \geq 0$ with an explicit constant $0<b_{2m}\leq 1$ using the Rellich identity, as in the proof of the Quantum Uniquely Ergodic Restriction (QUER) theorem of \cite{ctz}.

We first deduce from \eqref{eqhaha} that (Corollary \ref{lower})
\[
\liminf_{n \to \infty} \int_\beta \left|\psi_n(s)\right|^2 ds \geq 2 \int_\beta f^2(s) ds,
\]
and so
\[
\limsup_{n \to \infty} 2 \int_\beta f^2(s) ds\left(\int_\beta \left|\psi_n(s)\right|^2 ds \right)^{-1}\leq 1.
\]
Assume for simplicity that, for some $0 \leq a\leq 1$, we have
\[
\lim_{n \to \infty} 2 \int_\beta f^2(s) ds\left(\int_\beta \left|\psi_n(s)\right|^2 ds \right)^{-1} = a.
\]
Then \eqref{eqhaha} implies that
\[
\lim_{n \to \infty} \int_\beta \left|\frac{1}{\lambda_{n}^m}\frac{\partial^m \psi_n}{\partial s^m}(s)\right|^2 ds \left(\int_\beta \left|\psi_n(s)\right|^2 ds \right)^{-1}= (1-a) + a b_{2m}
\]
and we may apply the argument to
\[
h_n(\xi) = \lambda_{n} \left|\widehat{\psi_n}(\lambda_{n} \xi)\right|^2\left(\int_\beta \left|\psi_n(s)\right|^2 ds \right)^{-1}
\]
to conclude that all but finitely many $u_{n}$ have at least one sign change on $\beta$, by verifying that the unique measure having $(1-a)+a b_{2m}$ as the $2m$th moment and $0$ as the $(2m+1)$th moment is not positive-definite for any given $0 \leq a \leq 1$. This implies that the number of sign changes of $u_{n}$ along $\mathrm{Fix}(\sigma)$ tends to $+\infty$ as $n \to \infty$ (Theorem \ref{thm2}).

\section{\texorpdfstring{$L^p$}{Lp} estimates for the restriction to a curve of derivatives of eigenfunctions}
Let $u$ be a Laplacian eigenfunction with the eigenvalue $\lambda$. Let $L$ be a degree $m$ linear differential operator on $M$, i.e., for any coordinate patch $(U,p)$ there exists smooth functions $a_\alpha \in C^\infty (\mathbb{R}^n)$ (in which $a_\alpha\not\equiv 0$ for some $\alpha$ with $|\alpha|=m$) %added a condition to be degree EQUAL to m
such that for any $\phi,\psi \in C_0^\infty (U)$ and for each $f \in C^\infty (M)$,
\[
\phi L(\psi f)=\phi p^* \sum_{|\alpha|\leq m} a_\alpha(x)\partial^\alpha \left(p^{-1}\right)^* (\psi f).
\]
Recall that
\begin{equation}\label{triv2}
\sup_{x \in M}\left|Lu(x)\right| =O\left(\lambda^{m+\frac{1}{2}}\right),
\end{equation}
which is a consequence of the generalization of remainder estimate for spectral function by Avakumovic--Levitan--H\"{o}rmander to that for the derivatives of spectral function\cite{xu}. Denoting by $\langle f, g\rangle_\beta = \int_\beta f(s)\overline{g(s)}ds$, \eqref{triv2} implies that
\begin{equation}\label{triv1}
|\langle L u, u\rangle_{\beta}| \ll_\beta \sup_{x \in M}\left|Lu(x)\right|\sup_{x \in M}\left|u(x)\right|= O\left(\lambda^{m+1}\right)
\end{equation}

In the proof of  Theorem \ref{awe}, we need an improvement over \eqref{triv1}, and we achieve an improvement by combining the $L^2$ eigenfunction restrictions estimates along curves due to Burq, G{\'e}rard, and Tzvetkov \cite{bgt} and \eqref{triv2}.
\begin{lemma}\label{lem}
For any fixed degree $m$ differential operator $L$, we have
\begin{equation*}
|\langle L u, u\rangle_{\beta}| =O\left(\lambda^{m+\frac{3}{4}}\right).
\end{equation*}
\end{lemma}
\begin{proof}
By H{\"o}lder's inequality,
\[
|\langle L u, u\rangle_{\beta}| \leq \sup_{x \in M}\left|Lu(x)\right| \|u\|_{L^1(\beta)}.
\]
From \cite{bgt}, we have $\|u\|_{L^2(\beta)}=O\left(\lambda^{\frac{1}{4}}\right)$, hence
\[
\|u\|_{L^1(\beta)}\leq l(\beta)^{\frac{1}{2}}\|u\|_{L^2(\beta)} = O_{\beta} \left(\lambda^{\frac{1}{4}}\right).
\]
Therefore by \eqref{triv2}, we conclude
\[
|\langle L u, u\rangle_{\beta}|= O\left(\lambda^{m+\frac{3}{4}}\right).\qedhere
\]
\end{proof}
Since we only need any power saving over $O\left(\lambda^{m+1}\right)$ in \eqref{triv1} in our proof, it is unnecessary to optimize our bound in Lemma \ref{lem}. The optimal upper bound for $\|Lu\|_{L^2(\beta)}$ is $O\left(\lambda^{m+\frac{1}{4}}\right)$ which is sharp when $L=1$ and $M$ is the standard sphere $S^2$. Note that when $L$ corresponds to a normal derivative along $\beta$, the bound can be improved to $O(1)$ using second-microlocalization techniques, due to \cite{cth}.
\section{Rellich type analysis when QUE holds: even eigenfunctions}
In this section, we prove \eqref{eqhaha} with explicit constants using the Rellich identity assuming QUE.
\begin{theorem}\label{awe}
Assume that QUE holds for the sequence of even eigenfunctions $\{u_{n}\}$. Fix a segment $\beta \subset\mathrm{Fix}(\tau)$. For any fixed real valued function $f \in C_0^\infty (\beta)$ and for any fixed non-negative integer $m$, we have
\begin{align*}
&\lim_{k \to \infty}\int_\beta |f(t)u_{n}(t)|^2 dt - \lambda_{n}^{-2m}\int_{\beta} \left|\partial_t^{m}(f(t)u_{n}(t))\right|^2 dt\\
=&2\left(1-\frac{1}{\pi}\int_{-1}^1 \xi^{2m} \frac{d\xi}{\sqrt{1-\xi^2}}\right)\int_\beta f^2(t) dt.
\end{align*}
\end{theorem}
\begin{proof}
We drop the subscript $n$ in $u_n$ and $\lambda_n$ for simplicity.

Let $(t,n)$ be Fermi normal coordinates in a small tubular neighborhood $U_\epsilon$ of $\beta$ near a point $x_0 \in \beta$. Let $p:U_\epsilon \to \mathbb{R}^2$ be the coordinate chart. We may assume that
\[
U=U_\epsilon = p^{-1}\left(\{(t,n)~|~t\in V,~|n|<\epsilon\}\right)
\]
in these coordinates, where $V \subset \mathbb{R}$ is a coordinate chart that contains $x_0$. Let $(t,n,\xi_t,\xi_n)$ be the local coordinates of $T^*(U)$ under the identification
\begin{align*}
 \mathbb{R}^2& \to  T_{x=(t,n)}^*(U)\\
 (\xi_t,\xi_n) &\mapsto \xi_t dt+\xi_n dn.
\end{align*}
We consider the standard quantization in this coordinates, i.e., for any given symbol $a(t,n,\xi_t,\xi_n)$ of finite order, we let
\begin{multline*}
\mathrm{Op}(a)u(t_0,n_0)\\
=\frac{\lambda^n}{(2\pi)^n}\int_{p(U)\times \mathbb{R}^2} e^{\lambda i ((t_0-t)\xi_t+(n_0-n)\xi_n)}a(t_0,n_0,\xi_t,\xi_n)u(t,n) dtdnd\xi_t d\xi_n.
\end{multline*}
For example if $a(t,n,\xi_t,\xi_n)=\sum_{|\alpha|\leq N}a_\alpha(t,n)\xi_t^{\alpha_1}\xi_n^{\alpha_2}$, then
\[
\mathrm{Op}(a)u=\sum_{|\alpha|\leq N} a_\alpha(t,n)\left(\frac{\partial_t}{i\lambda}\right)^{\alpha_1}\left(\frac{\partial_n}{i\lambda}\right)^{\alpha_2}u .
\]
Let $U_{-}\subset U$ be given by
\[
U_{-}=\{(t,n) \in U~|~ n<0\}.
\]
For any pseudo-differential operator $T$ on $M$, from Green's formula,  we have
\begin{equation}\label{green}
\langle \Delta_g T u,u\rangle_{U_-}-\langle T u, \Delta_g  u\rangle_{U_-}=\langle \partial_n T u|_\beta,u|_\beta\rangle_\beta -\langle T u|_\beta, \partial_n  u|_\beta\rangle_\beta
\end{equation}
Since $u$ is an eigenfunction, $\langle T u, \Delta_g  u\rangle_{U_-}=\langle T \Delta_g u,   u\rangle_{U_-}$. Also since we are assuming that $u$ is even, $\langle T u|_\beta, \partial_n  u|_\beta\rangle_\beta=0$. Therefore we have the Rellich identity:
\begin{equation}\label{rel}
\frac{1}{\lambda}\langle \lbrack-\Delta_g, T\rbrack u,u\rangle_{U_-} = -\frac{1}{\lambda}\langle \partial_n T u|_\beta,u|_\beta\rangle_\beta,
\end{equation}
where $\lambda^{-1}$ is the normalizing factor.

Now fix $\chi \in C_0^\infty(\mathbb{R})$ such that
\begin{equation*}
\chi(x)= \left\{
\begin{array}{cl} 0 & \hspace{10mm}\text{if } |x|\geq 1,\\
1 & \hspace{10mm}  \text{if }|x|\leq \frac{1}{2}.
\end{array} \right.
\end{equation*}
We define a symbol supported near $\beta$ for $0<\delta<\epsilon$,
\[
a_{\delta,m}(x,\xi) = \chi\left(\frac{n}{\delta}\right)f^2(t) i\xi_n \sum_{k=0}^{m-1}\xi_t^{2k}
\]
and let $T=\mathrm{Op}(a_{\delta,m})$. Observing that $-(\partial_n^2+\partial_t^2)u = \lambda^2 u$ along $\beta$, we may rewrite the RHS of \eqref{rel} as
\[
\left\langle f^2\left(1+(-1)^{m-1}\lambda^{-2m}\partial_t^{2m}\right)u|_\beta, u|_\beta \right\rangle_{\beta}.
\]
We use integrate by parts to further simplify the second term as follows
\begin{align*}
&\left\langle (-1)^{m-1}f^2\lambda^{-2m}\partial_t^{2m}u|_\beta, u|_\beta \right\rangle_{\beta}+\lambda^{-2m}\left\langle \partial_t^{m}\left(fu|_\beta\right), \partial_t^{m}\left(fu|_\beta\right) \right\rangle_{\beta}\\
=&\left\langle (-1)^{m-1}\lambda^{-2m}f\lbrack f,\partial_t^{2m}\rbrack u|_\beta, u|_\beta \right\rangle_{\beta}\\
=&O_{m,f}(\lambda^{-\frac{1}{4}}),
\end{align*}
where we used Lemma \ref{lem} with $L=f\lbrack f,\partial_t^{2m}\rbrack$ in the last estimate. So we have
\begin{equation}\label{RHS}
-\frac{1}{\lambda^2}\langle \partial_n T u|_\beta,u|_\beta\rangle_\beta=\int_\beta |f(t)u(t)|^2 dt - \lambda^{-2m}\int_{\beta} \left|\partial_t^{m}(f(t)u(t))\right|^2 dt+O_{m,f}\left(\lambda^{-\frac{1}{4}}\right).
\end{equation}

Now let $-\Delta_g=\sum_{|\alpha|=1,2}b_\alpha(n,t)\left(\frac{\partial_n}{i}\right)^{\alpha_1}\left(\frac{\partial_t}{i}\right)^{\alpha_2}$.
Observe from \eqref{triv2} that if $\alpha_1+\alpha_2=1$,
\[
\frac{1}{\lambda} \left\lbrack b_\alpha(n,t) \left(\frac{\partial_n}{i}\right)^{\alpha_1}\left(\frac{\partial_t}{i}\right)^{\alpha_2},\chi\left(\frac{n}{\delta}\right)f^2(t) \frac{(-1)^k \partial_n \partial_t^{2k}}{\lambda^{2k+1}}\right\rbrack u= O_{\delta,m,f}(\lambda^{-\frac{1}{2}})
\]
and that if $\alpha_1+\alpha_2=2$,
\begin{align*}
&\frac{1}{\lambda} \left\lbrack b_\alpha(n,t) \left(\frac{\partial_n}{i}\right)^{\alpha_1}\left(\frac{\partial_t}{i}\right)^{\alpha_2},\chi\left(\frac{n}{\delta}\right)f^2(t) \frac{\partial_n (-1)^k\partial_t^{2k}}{\lambda^{2k+1}}\right\rbrack u\\
= &\mathrm{Op}\left(\frac{\alpha_1}{\delta}b_\alpha(n,t)\chi'\left(\frac{n}{\delta}\right)f^2(t)\xi_n^{\alpha_1}\xi_t^{2k+\alpha_2}\right)u\\
+&\mathrm{Op}\left(\chi\left(\frac{n}{\delta}\right)R_{m,f,\alpha}(n,t,\xi_n,\xi_t)\right)u+O_{\delta,m,f}(\lambda^{-\frac{1}{2}})
\end{align*}
for some symbol $R_{m,f,\alpha}$ of finite order depending only on $m,f,\alpha$. Therefore we may reexpress the LHS of \eqref{rel} as
\begin{align*}
&\left\langle \mathrm{Op}\left(\sum_{|\alpha|=2}\frac{\alpha_1\xi_n^{\alpha_1}\xi_t^{\alpha_2}}{\delta}b_\alpha(n,t)\chi'\left(\frac{n}{\delta}\right)f^2(t)\sum_{k=0}^{m-1}\xi_t^{2k}\right)u,u\right\rangle_{U_-}\\
+&\left\langle \mathrm{Op}\left(\chi\left(\frac{n}{\delta}\right)R_{m,f}(n,t,\xi_n,\xi_t)\right)u,u\right\rangle_{U_-}\\
+&O_{\delta,m,f}(\lambda^{-\frac{1}{2}})
\end{align*}
for some finite order symbol $R_{m,f}$.

We bound the second inner product using Cauchy-Schwartz inequality by
\begin{align*}
&\left|\left\langle \mathrm{Op}\left(\chi\left(\frac{n}{\delta}\right)R_{m,f}(n,t,\xi_n,\xi_t)\right) u, u\right\rangle_{U_-}\right| \\
\leq &\left\|\mathrm{Op}\left(\chi\left(\frac{n}{\delta}\right)R_{m,f}(n,t,\xi_n,\xi_t)\right) u\right\|_{L^2(U_-)}^2\\
\leq &\left\|\mathrm{Op}\left(\chi\left(\frac{n}{\delta}\right)R_{m,f}(n,t,\xi_n,\xi_t)\right) u\right\|_{L^2(U)}^2
\end{align*}
and from the assumption that the QUE holds, we may estimate the last quantity as $O_{m,f}(\delta)+o_{\delta,m,f}(1)$ as $\lambda$ tends to $+\infty$.

Now let $\chi_0\in C_0^\infty (\mathbb{R})$ be given by $\chi_0(x)=\chi'(x)$ if $x<0$, and $\chi_0(x)=0$ otherwise. We then have
\begin{align}\label{LHS}
&\left\langle \mathrm{Op}\left(\sum_{|\alpha|=2}\frac{\alpha_1\xi_n^{\alpha_1}\xi_t^{\alpha_2}}{\delta}b_\alpha(n,t)\chi'\left(\frac{n}{\delta}\right)f^2(t)\sum_{k=0}^{m-1}\xi_t^{2k}\right)u,u\right\rangle_{U_-}\notag\\
=&\left\langle \mathrm{Op}\left(\sum_{|\alpha|=2}\frac{\alpha_1\xi_n^{\alpha_1}\xi_t^{\alpha_2}}{\delta}b_\alpha(n,t)\chi_0\left(\frac{n}{\delta}\right)f^2(t)\sum_{k=0}^{m-1}\xi_t^{2k}\right)u,u\right\rangle_{U}\notag\\
=&\int_{S^*U} \sum_{|\alpha|=2}\frac{\alpha_1\xi_n^{\alpha_1}\xi_t^{\alpha_2}}{\delta}b_\alpha(n,t)\chi_0\left(\frac{n}{\delta}\right)f^2(t)\sum_{k=0}^{m-1}\xi_t^{2k}d\mu+o_{\delta,m,f}(1)
\end{align}
as $\lambda$ tends to $+\infty$ from the assumption that QUE holds.

We therefore conclude from \eqref{RHS} and \eqref{LHS} that
\begin{multline*}
\int_\beta |f(t)u(t)|^2 dt - \lambda^{-2m}\int_{\beta} \left|\partial_t^{m}(f(t)u(t))\right|^2 dt\\
=\int_{S^*U} \sum_{|\alpha|=2}\frac{\alpha_1\xi_n^{\alpha_1}\xi_t^{\alpha_2}}{\delta}b_\alpha(n,t)\chi_0\left(\frac{n}{\delta}\right)f^2(t)\sum_{k=0}^{m-1}\xi_t^{2k}d\mu\\
+O_{m,f}\left(\lambda^{-\frac{1}{4}}\right)+O_{m,f}(\delta)+o_{\delta,m,f}(1)
\end{multline*}
and so
\begin{multline*}
\lim_{\lambda \to \infty}\int_\beta |f(t)u(t)|^2 dt - \lambda^{-2m}\int_{\beta} \left|\partial_t^{m}(f(t)u(t))\right|^2 dt\\
=\int_{S^*U} \sum_{|\alpha|=2}\frac{\alpha_1\xi_n^{\alpha_1}\xi_t^{\alpha_2}}{\delta}b_\alpha(n,t)\chi_0\left(\frac{n}{\delta}\right)f^2(t)\sum_{k=0}^{m-1}\xi_t^{2k}d\mu+O_{m,f}(\delta).
\end{multline*}
Note that no terms in the left hand side depend on $\delta$. Also note that $b_{20}(0,t)=b_{02}(0,t)=1$ and $b_{11}(0,t)=0$ since we are taking Fermi normal coordinate. Therefore by taking $\delta \to 0$, we have
\begin{align*}
&\lim_{\delta \to 0}\int_{S^*U} \sum_{|\alpha|=2}\frac{\alpha_1\xi_n^{\alpha_1}\xi_t^{\alpha_2}}{\delta}b_\alpha(n,t)\chi_0\left(\frac{n}{\delta}\right)f^2(t)\sum_{k=0}^{m-1}\xi_t^{2k}d\mu+O_{m,f}(\delta)\\
=&\int_{S_\beta^*U} \sum_{|\alpha|=2}\alpha_1 b_\alpha(0,t)f^2(t)\xi_n^{\alpha_1}\xi_t^{\alpha_2}\sum_{k=0}^{m-1}\xi_t^{2k}d\mu\\
=&\frac{1}{\pi} \int_\beta f^2(t) dt \int_{\xi_t^2+\xi_n^2=1}(1-\xi_t^{2m}) d\xi\\
=&2\left(1-\frac{1}{\pi}\int_{-1}^1 \xi^{2m} \frac{d\xi}{\sqrt{1-\xi^2}}\right)\int_\beta f^2(t) dt. \end{align*}
This implies that
\begin{multline*}
\lim_{\lambda \to \infty}\int_\beta |f(t)u(t)|^2 dt - \lambda^{-2m}\int_{\beta} \left|\partial_t^{m}(f(t)u(t))\right|^2 dt\\
=2\left(1-\frac{1}{\pi}\int_{-1}^1 \xi^{2m} \frac{d\xi}{\sqrt{1-\xi^2}}\right)\int_\beta f^2(t) dt+o_{m,f}(1)
\end{multline*}
as $\delta \to 0$, and since $\delta$ can be chosen arbitrarily small, we conclude that
\begin{multline*}
\lim_{\lambda \to \infty}\int_\beta |f(t)u(t)|^2 dt - \lambda^{-2m}\int_{\beta} \left|\partial_t^{m}(f(t)u(t))\right|^2 dt\\
=2\left(1-\frac{1}{\pi}\int_{-1}^1 \xi^{2m} \frac{d\xi}{\sqrt{1-\xi^2}}\right)\int_\beta f^2(t) dt 
\end{multline*}
\end{proof}

As an immediate application of Theorem \ref{awe}, we give a sharp lower bound for the $L^2$ estimate of the restriction of eigenfunctions.
\begin{corollary}\label{lower}
Assume that QUE holds for the sequence of even eigenfunctions $\{u_{n}\}$. Then for any fixed real valued function $f \in C_0^\infty (\beta)$, we have
\begin{equation*}
\liminf_{n \to \infty} \int_{\beta} f^2(t)|u_{n}(t)|^2 dt \geq 2\int_\beta f^2(t) dt.
\end{equation*}
\end{corollary}
\begin{proof}
By the positivity of $\lambda^{-2m}\int_{\beta} \left|\partial_t^{m}(f(t)u(t))\right|^2 dt$,
\begin{align*}
\liminf_{n \to \infty} \int_{\beta} f^2(t)|u_{n}(t)|^2 dt&\geq 2\left(1-\frac{1}{\pi}\int_{-1}^1 \xi^{2m} \frac{d\xi}{\sqrt{1-\xi^2}}\right)\int_\beta f^2(t) dt\\
&=2\int_\beta f^2(t) dt +O(1/m).
\end{align*}
Since the limit does not depend on $m$, we conclude that
\[
\liminf_{n \to \infty} \int_{\beta} f^2 (t)|u_{n}(t)|^2 dt \geq 2\int_\beta f^2(t) dt.\qedhere
\]
\end{proof}
\begin{remark}
Constant lower bound for $L^2$ norm of the restriction of an eigenfunction to a geodesic segment is first proven in \cite{GRS}, from the arithmetic QUE theorem \cite{lin06,so10}.
\end{remark}
\begin{remark}
If the geodesic flow on $M$ is ergodic, it is known that there exists a density $1$ subsequence $\{u_{n}\}$ of even eigenfunctions that satisfies
\begin{equation}\label{fac2}
\lim_{n \to \infty} \|u_{n}\|_{L^2(\beta)}^2 = 2 l(\beta),
\end{equation}
hence the lower bound in Corollary \ref{lower} is sharp. Existence of such a subsequence is a consequence of results which are studied in \cite{burq,tz1,dz,ctz}.
\end{remark}
\begin{remark}
If $\beta$ is not a part of $\mathrm{Fix}(\tau)$ and satisfies a certain asymmetry condition (see, for instance, \cite[Definition 1]{tz1}), then
\[
\lim_{k \to \infty} \|u_{n_k}\|_{L^2(\beta)}^2 = l(\beta)
\]
along a density $1$ subsequence $\{u_{n_k}\}$ of $\{u_n\}$. When $\beta$ is a segment of $\mathrm{Fix}(\tau)$, then every odd eigenfunctions vanish identically on $\beta$, hence explaining why we expect the factor $2$ in \eqref{fac2}.
\end{remark}
\section{The number of nodal domains of even eigenfunctions}
\subsection{Graph structure of the nodal set and Euler's inequality}\label{sec1}
In this section we briefly review the topological argument in \cite{GRS,JZ1} on bounding the number of nodal domains from below by the number of zeros on $\mathrm{Fix}(\tau)$. We refer the readers to \cite{JZ1} for details.

Firstly note that if there exists a segment of $\eta \subset \mathrm{Fix}(\tau)$ such that $\eta \subset Z_u$, then because normal derivative of $u$ vanishes along $\mathrm{Fix}(\tau)$, any point on $\eta$ is a singular point, contradicting the upper bound on the number of singular points in \cite{d}. Therefore together with the following lemma on local structure of the nodal set, we conclude that $Z_u \cap \mathrm{Fix}(\tau)$ is a finite set of points.
\begin{lemma}[Section 6.1, \cite{JZ1}]\label{graph}
Assume that $u$ vanishes to order $N$ at $x_0$. Then there exists a small neighborhood $U$ of $x_0$ such that the nodal set in $U$ is $C^1$ equivalent to $2N$ equi-angular rays emanating from $x_0$.
\end{lemma}
From Lemma \ref{graph}, we may view the nodal set as a graph (a {\it nodal graph}) embedded on a surface as follows.
\begin{enumerate}
\item For each connected component of $Z_u$ that is homeomorphic to a circle and that does not intersect $\mathrm{Fix}(\tau)$, we add a vertex.
\item Each singular point is a vertex.
\item Each intersection point in $\mathrm{Fix}(\tau) \cap Z_u$ is a vertex.
\item Edges are the arcs of $Z_u \cup \mathrm{Fix}(\tau)$ that join the vertices listed above.
\end{enumerate}
Let $V(u)$ and $E(u)$ be the finite set of vertices and the finite set of edges given above, respectively. This way, we obtain a {\it nodal graph} $V(u),E(u)$ of $u$ embedded into the surface $M$.

From the assumption that $\mathrm{Fix}(\tau)$ is separating, the nodal domains that intersect $\mathrm{Fix}(\tau)$ are cut in two by $\mathrm{Fix}(\tau)$. Therefore the number of faces divided by two bounds the number of nodal domains $\mathcal{N}(u)$ from below.

Observe from Lemma \ref{graph} that every vertex of a nodal graph has degree at least $2$. Then by Euler's inequality \cite[(6.1)]{JZ1}:
\[
|V(u)|-|E(u)|+|F(u)| -m(u) \geq 1- 2 \mathfrak{g},
\]
we obtain a lower bound for the number of nodal domains by the number of zeros on $\mathrm{Fix}(\tau)$. Here $m(G)$ is the number of connected components of the nodal graph and $\mathfrak{g}$ is the genus of the surface $M$.
\begin{lemma}[Lemma 6.4, \cite{JZ1}]
\[
\mathcal{N}(u) \geq \frac{1}{2}\#\left(Z_{u} \cap \mathrm{Fix}(\tau) \right)+1-\mathfrak{g}.
\]
\end{lemma}
Therefore in order to prove Theorem \ref{thm1}, it is sufficient to prove the following theorem.
\begin{theorem}\label{thm2}
Assume that QUE holds for the sequence of even eigenfunctions $\{u_{n}\}_{n \geq 1}$. Then
\[
\lim_{n\to \infty}\#\left(Z_{u_{n}}\cap \mathrm{Fix}(\tau)\right) = +\infty.
\]
\end{theorem}
\subsection{Lemmata from probability theory}
In order to prove Theorem \ref{thm2}, we first recall some facts about probability measures. We assume that all random variables in this section are defined on the real line.

\begin{lemma}\label{prob1}
Suppose a random variable $X$ has moments $\mu_k=\mathbb{E}[X^k]$ that satisfies the condition
\[
\limsup_{k\rightarrow\infty}\mu_{2k}^{\frac{1}{2k}}/2k=r<\infty.
\]
Then, $X$ has the unique distribution with moments $(\mu_k)_{k\geq1}$.
\end{lemma}
\begin{proof}
See \cite[Theorem 3.3.11]{durrett2010probability}.
\end{proof}

\begin{lemma}\label{prob2}
If $X_n$ converges to $X$ in moments and the distribution of $X$ is uniquely determined by its moments, then for each $t\in\mathbb{R}$, $\mathbb{E}[e^{itX_n}]$ converges to $\mathbb{E}[e^{itX}]$.
\end{lemma}
\begin{proof}
Suppose we have a counterexample of this lemma. That is, we have a sequence $(X_n)$ of random variables and $X$ a random variable, such that $\mathbb{E}\lbrack X_n^m\rbrack \rightarrow\mathbb{E}\lbrack X^m\rbrack $ for all $m>0$, but $\mathbb{E}\lbrack \exp(it_0X_n)\rbrack \not\rightarrow\mathbb{E}\lbrack \exp(it_0 X)\rbrack $ for some $t_0\in\mathbb{R}$.

Let $F_n(x):=\Pr\lbrack X_n\leq x\rbrack $ be the cumulative distribution functions of random variable $X_n$. By Helly's selection theorem \cite[Theorem 3.2.6]{durrett2010probability}, together with the tightness of $(X_{n})$'s \cite[Theorem 3.2.7 and 3.2.8]{durrett2010probability} there exists a subsequence $(F_{n_k})$ that converges to a cumulative distribution function $G$ of some random variable $Y$ on the real line.

\cite[Theorem 3.2.2]{durrett2010probability} implies that, by appropriately settling the probability space $\Omega$ for $X_{n_k}$'s and $Y$, we can have $X_{n_k}(\omega)\rightarrow Y(\omega)$ almost surely for $\omega\in\Omega$. (For instance, we can set $\Omega=(0,1)$, $\Pr=($the Lebesgue measure$)$, and $X_{n_k}(\omega)=\sup\{x\in\mathbb{R}~|~F_{n_k}(x)<\omega\}$, etc.) In particular, as $\exp(it_0X_{n_k})\rightarrow\exp(it_0Y)$ almost surely, together with $|\exp(it_0X_{n_k})|\leq 1$ for all $k$ implies that $\mathbb{E}\lbrack \exp(it_0X_{n_k})\rbrack \rightarrow\mathbb{E}\lbrack \exp(it_0Y)\rbrack $ by the bounded convergence theorem.

From the assumption that $X$ is the unique random variable with the sequence of the moments $\left(\mathbb{E}[X^m]\right)$, we claim $X=Y$ by showing that $\mathbb{E}\lbrack X^m\rbrack =\mathbb{E}\lbrack Y^m\rbrack $ for all $m>0$. Equivalently, $\mathbb{E}\lbrack X_{n_k}^m\rbrack \rightarrow\mathbb{E}\lbrack Y^m\rbrack $ as $k\rightarrow\infty$. Denote by $\chi_M$ the indicator function of $\lbrack -M,M\rbrack $ for $M>0$. We first estimate:
\begin{eqnarray*}
|\mathbb{E}\lbrack X_{n_k}^m\rbrack -\mathbb{E}\lbrack Y^m\rbrack | &\leq& \mathbb{E}\lbrack |X_{n_k}|^m(1-\chi_M(X_{n_k}))\rbrack +\mathbb{E}\lbrack |Y|^m(1-\chi_M(Y))\rbrack  \\
&&\quad+|\mathbb{E}\lbrack X_{n_k}^m\chi_M(X_{n_k})-Y^m\chi_M(Y)\rbrack |.
\end{eqnarray*}
We bound the first term by Cauchy--Schwarz inequality and Markov inequality,
\begin{eqnarray*}
\mathbb{E}\lbrack |X_{n_k}|^m(1-\chi_M(X_{n_k}))\rbrack  &\leq& \mathbb{E}\lbrack |X_{n_k}|^{2m}\rbrack ^{\frac{1}{2}}\mathbb{E}\lbrack (1-\chi_M(X_{n_k}))^2\rbrack ^{\frac{1}{2}} \\
&=& \mathbb{E}\lbrack X_{n_k}^{2m}\rbrack ^{\frac{1}{2}}(\Pr\lbrack |X_{n_k}|>M\rbrack )^{\frac{1}{2}} \\
&\leq& \mathbb{E}\lbrack X_{n_k}^{2m}\rbrack ^{\frac{1}{2}}\mathbb{E}\lbrack X_{n_k}^2\rbrack ^{\frac{1}{2}}M^{-1} \leq KM^{-1},
\end{eqnarray*}
where $K=\sup\{\mathbb{E}\lbrack X_{n_k}^{2m}\rbrack ,\mathbb{E}\lbrack X_{n_k}^2\rbrack ~|~k\}<\infty$. We bound the second term by Fatou's lemma,
\[
\mathbb{E}\lbrack |Y|^m(1-\chi_M(Y))\rbrack  \leq \liminf_{k\rightarrow\infty}\mathbb{E}\lbrack |X_{n_k}|^m(1-\chi_M(X_{n_k}))\rbrack  \leq KM^{-1},
\]
where we used the estimate of the first term in the the last inequality. Finally, observe that the third term converges to $0$, i.e., $|\mathbb{E}\lbrack X_{n_k}^m\chi_M(X_{n_k})-Y^m\chi_M(Y)\rbrack |\rightarrow 0$ as $k\rightarrow\infty$, by the bounded convergence theorem.

Therefore
\[
\limsup_{k \to \infty}|\mathbb{E}\lbrack X_{n_k}^m\rbrack -\mathbb{E}\lbrack Y^m\rbrack | =O(M^{-1})
\]
and since $M$ can be chosen arbitrarily large, we conclude $\mathbb{E}\lbrack X_{n_k}^m\rbrack \rightarrow\mathbb{E}\lbrack Y^m\rbrack $ which implies that $X=Y$. Therefore $\mathbb{E}\lbrack \exp(it_0X_{n_k})\rbrack \rightarrow\mathbb{E}\lbrack \exp(it_0Y)\rbrack =\mathbb{E}\lbrack \exp(it_0X)\rbrack $, contradicting the initial assumption
\[
\mathbb{E}\lbrack \exp(it_0X_{n_k})\rbrack \not\rightarrow\mathbb{E}\lbrack \exp(it_0X)\rbrack.\qedhere
\]
\end{proof}

We now present a new method for detecting sign changes of functions using Lemma \ref{prob1}, Lemma \ref{prob2} and Bochner's theorem.
\begin{lemma}\label{prob3}
Let $\{f_n\}$ be a sequence of real valued functions in $C_0^\infty([0,1])$ and let $\{a_n\}$ be a sequence of positive reals such that for each fixed nonnegative integer $m$ we have
\begin{equation}\label{asdf}
\lim_{n \to \infty} a_n^{-2m}\int_0^1 \left|\partial_x^m f_n(x)\right|^2 dx = b_{2m}
\end{equation}
for some positive real numbers $b_{2m}$. Assume that $d\mu(\xi)$ is the unique probability distribution whose $2m$th moment is $b_{2m}/b_0$ and whose $(2m+1)$th moment is zero for any $m \geq 0$. If $d\mu(\xi)$ is not positive-definite, then all but finitely many $f_n$ has at least one sign change on $(0,1)$.
\end{lemma}
\begin{proof}
Assume for contradiction that there exists a subsequence $\{f_{n_k}\}$ such that $f_{n_k}$ does not change sign on $(0,1)$ for all $k \geq 1$. Let $h_k$ be given by
\[
h_k(\xi)= \frac{a_{n_k}}{2\pi}\left|\int_0^1 f_{n_k}(x) e^{ia_{n_k} \xi x}dx\right|^2 \left(\int_0^1 |f_{n_k}(x)|^2 dx\right)^{-1}.
\]
Then from \eqref{asdf}, we have for each $m \geq 0$,
\[
\lim_{k \to \infty} \int_{-\infty}^\infty \xi^{2m} h_k(\xi) d\xi = b_{2m}/b_0
\]
and since $h_k(\xi)$ is an even function in $\xi$, the sequence of probability distribution $\{h_k(\xi)d\xi \}$ converges in moments to $d\mu(\xi)$. We therefore conclude from Lemma \ref{prob2} that the sequence of characteristic functions of $h_k(\xi)d\xi$ converges point-wise to the characteristic function of $d\mu(\xi)$.

Now observe that since $f_{n_k}$ does not change sign along $(0,1)$, $h_k(\xi)$ is a positive-definite function in $\xi$ for each $k$ by Bochner's theorem. Therefore the characteristic function of $h_k(\xi)d\xi$ is a non-negative function for each $k$. However, since we assumed $d\mu(\xi)$ is not positive-definite, the characteristic function $\int_{-\infty}^\infty e^{it\xi}d\mu(\xi)$ is negative for some $t \in \mathbb{R}$, which contradicts the point-wise convergence of characteristic functions. We therefore conclude that all but finitely many $f_n$ has at least one sign change on $(0,1)$.
\end{proof}
\subsection{Sign changes of even eigenfunctions on fixed segments}
\begin{lemma}\label{lem1}
Assume that QUE holds for the sequence of even eigenfunctions $\{u_{n}\}_{n \geq 1}$. For any fixed segment $\beta \subset \mathrm{Fix}(\tau)$, all but finitely many $u_{n}$ have at least one sign change on $\beta$.
\end{lemma}
\begin{proof}
Assume for contradiction that there exists a subsequence of even eigenfunctions $\{u_{n_k}\}_{k \geq 1}$ such that $u_{n_k}$ does not change sign along $\beta$ for all $k \geq 1$. Fix a non-negative function $f \in C_0^\infty(\beta)$.

Firstly, by Corollary \ref{lower}, we can find a subsequence $\{u_{j_k}\}_{k \geq 1}\subset \{u_{n_k}\}_{k \geq 1}$ such that
\begin{equation*}
\lim_{k \to \infty}2 \int_\beta f^2(t) dt||fu_{j_k}||_{L^2(\beta)}^{-2} = a
\end{equation*}
for some $0\leq a\leq 1$. Then by Theorem \ref{awe}, we have that
\[
\lim_{k \to \infty}1 - \lambda_{j_k}^{-2m}\int_{\beta} \left|\partial_t^{m}(f(t)u_{j_k}(t))\right|^2 ||fu_{j_k}||_{L^2(\beta)}^{-2}dt=a-\frac{a}{\pi}\int_{-1}^1 \xi^{2m} \frac{d\xi}{\sqrt{1-\xi^2}}.
\]
We therefore have
\begin{align*}
\lim_{k \to \infty}\lambda_{j_k}^{-2m}\int_{\beta} \left|\partial_t^{m}(f(t)u_{j_k}(t))\right|^2 ||fu_{j_k}||_{L^2(\beta)}^{-2}dt &= (1-a)+\frac{a}{\pi}\int_{-1}^1 \xi^{2m} \frac{d\xi}{\sqrt{1-\xi^2}}\\
&=\int_{-\infty}^\infty\xi^{2m} d\mu_a(\xi)\\
&=b_{2m},
\end{align*}
where $d\mu_a(\xi)$ is the probability measure given by
\[
d\mu_a(\xi)=\frac{(1-a)}{2}(\delta_{-1}(\xi)+ \delta_{1}(\xi))d\xi+\frac{a}{\pi}I_{[-1,1]}(\xi)\frac{d\xi}{\sqrt{1-\xi^2}}.
\]
Here $I_{[-1,1]}(\xi)$ is the indicator function of $[-1,1]$. Observe that
\[
\int_{-\infty}^\infty \xi^{2m} d\mu_a(\xi) \leq  (1-a)+\frac{a}{\pi}\int_{-1}^1 \frac{d\xi}{\sqrt{1-\xi^2}} =1,
\]
and
\[
\limsup_{m \to \infty}\frac{1}{2m} b_{2m}^{\frac{1}{2m}} = 0 < +\infty,
\]
so by Lemma \ref{prob1}, $d\mu_a(\xi)$ is the only probability measure on $\mathbb{R}$ whose $2m$th moment is $b_{2m}$ and whose $(2m+1)$th moment is zero for any $m \geq 0$.

Note that the characteristic function of $d\mu_a(\xi)$ is given by
\[
\int_{-\infty}^\infty e^{it\xi} d\mu_a(\xi)= (1-a) \cos (t)+ a J_0(t),
\]
where $J_0(t)$ is the Bessel function of the first kind and that
\[
(1-a) \cos (\pi)+ a J_0(\pi) \leq J_0(\pi) = -0.3042\ldots <0
\]
which implies that $d\mu_a(\xi)$ is not positive-definite.

It now follows from Lemma \ref{prob3} that $f(t)u_{j_k}$ has at least one sign change along $\beta$ for all but finitely many $k$, which contradicts the assumption that $u_{n_k}$ does not change sign on $\beta$ for all $k\geq 1$. We therefore conclude that all but finitely many $u_{n}$ have at least one sign change on $\beta$.
\end{proof}

We complete the proof of Theorem \ref{thm1} by proving Theorem \ref{thm2}.
\begin{proof}[Proof of Theorem \ref{thm2}]
Fix $N \in \mathbb{N}$. Let $\beta_1, \cdots , \beta_N \subset \mathrm{Fix}(\tau)$ be a set of disjoint segments. Then by Lemma \ref{lem1}, for all sufficiently large $k$, $u_{n}$ has at least one sign change on each curve $\beta_i$ for $i=1,\cdots, N$. Hence we have
\[
\liminf_{k\to \infty}\#\left(Z_{u_{n}}\cap \mathrm{Fix}(\tau)\right) \geq N
\]
and since $N$ can be chosen arbitrarily large, we conclude that
\[
\lim_{k\to \infty}\#\left(Z_{u_{n}}\cap \mathrm{Fix}(\tau)\right) = +\infty. \qedhere
\]
\end{proof}
\section{Nodal domains of odd eigenfunctions}\label{last}
In this section we prove an analogy of Theorem \ref{thm1} for a sequence odd eigenfunctions assuming QUE. Recall from \eqref{green} that
\begin{equation*}
\langle \Delta_g T u,u\rangle_{U_-}-\langle T u, \Delta_g  u\rangle_{U_-}=\langle \partial_n T u|_\beta,u|_\beta\rangle_\beta -\langle T u|_\beta, \partial_n  u|_\beta\rangle_\beta.
\end{equation*}
From the assumption that $u$ is an odd eigenfunction, we have the Rellich identity for odd eigenfunctions
\begin{equation}\label{relodd}
\frac{1}{\lambda}\langle \lbrack-\Delta_g, T\rbrack u,u\rangle_{U_-}=\frac{1}{\lambda}\langle T u|_\beta, \partial_n  u|_\beta\rangle_\beta.
\end{equation}
Let
\[
a_{\delta,m}(x,\xi)=\chi\left(\frac{n}{\delta}\right)f^2(t) \xi_t^{2m}\xi_n,
\]
and let $T= \mathrm{Op}(a_{\delta,m})$. For simplicity, let $Nu=\lambda^{-1}\partial_n u|_\beta$. Then the RHS of \eqref{relodd} is
\begin{multline*}
(-1)^m\lambda^{-2m}\int_{\beta} f^2(t)\left(\partial_t^{2m}Nu(t) \right) Nu(t)dt\\
=\lambda^{-2m}\int_{\beta} \left|\partial_t^{m}\left(f(t)Nu(t) \right)\right|^2 dt+O_{m,f} \left(\lambda^{-\frac{1}{4}}\right)
\end{multline*}
and the LHS of \eqref{relodd} is
\[
\frac{2}{\pi}\int_{-1}^1 \xi^{2m}\sqrt{1-\xi^2} d\xi \int_\beta f^2(t) dt+o_{\delta,m,f}(1)+O_{m,f}(\delta).
\]
Therefore Theorem \ref{awe} for odd eigenfunctions assuming QUE is
\begin{align*}
\lim_{k \to \infty} \lambda_{n}^{-2m}\int_\beta |\partial_t^m \left(f(t)Nu_{n}(t)\right) |^2 dt &= \frac{2}{\pi} \int_\beta f^2(t)dt \int_{-1}^1 \xi^{2m}\sqrt{1-\xi^2}d\xi\\
&=b_{2m} \int_\beta f^2(t)dt .
\end{align*}
Let $d\mu(\xi)=\pi^{-1}\sqrt{1-\xi^2}d\xi$. Observe that
\[
b_{2m}= \frac{2}{\pi} \int_{-1}^1 \xi^{2m}\sqrt{1-\xi^2}d\xi \leq \frac{2}{\pi} \int_{-1}^1 \sqrt{1-\xi^2}d\xi=1,
\]
and
\[
\limsup_{m \to \infty } \frac{1}{2m}b_{2m}^{\frac{1}{2m}}=0<+\infty,
\]
so by Lemma \ref{prob1}, $d\mu(\xi)$ is the only probability measure on $\mathbb{R}$ whose $2m$th moment is $b_{2m}$ and whose $(2m+1)$th moment is zero for any $m\geq 0$. Now note that
\[
\int e^{it\xi} d\mu(\xi) = J_1(t)/t
\]
and since $J_1(5)/5=-0.0655\ldots <0$, $d\mu(\xi)$ is not positive-definite, so we may apply Lemma \ref{prob3} to conclude
\begin{lemma}\label{lem2}
Assume that QUE holds for a sequence of odd eigenfunctions $\{u_{n}\}$. For any fixed segment $\beta \subset \mathrm{Fix}(\tau)$, all but finitely many $\partial_n u_{n}|_\beta$ has at least one sign change on $\beta$.
\end{lemma}
%\begin{remark}
%The complication in the proof of the even eigenfunctions case occurs because we do not know whether $\|u\|_{L^2(\beta)}^2$ should necessarily converge to a constant when $\lambda \to \infty$. If so, the constant is $2l(\beta)$ and one can prove QUE for the restriction of even eigenfunctions:
%\[
%\lim_{\lambda \to \infty} \langle \mathrm{Op}_{\beta}(a)  u|_\beta, u|_\beta \rangle_\beta = 2 \int_{S_\beta^*M}  a_0(\xi_t, t)d\xi dt.
%\]
%Then the proof of Lemma \ref{lem1} can be simplified as above.
%\end{remark}
As in Theorem \ref{thm2}, Lemma \ref{lem2} implies the following.
\begin{theorem}
Assume that QUE holds for a sequence of odd eigenfunctions $\{u_{n}\}$. Then
\[
\lim_{k \to \infty} \#\{x \in \mathrm{Fix}(\tau) ~:~(\partial_n u_{n})(x)=0\} = +\infty.
\]
\end{theorem}
We now use the topological argument in \cite{GRS,JZ1} to conclude an analogy of Theorem \ref{thm1} for odd eigenfunctions.
\begin{theorem}
Assume that QUE holds for a sequence of odd eigenfunctions $\{u_{n}\}$. Then
\[
\lim_{k \to \infty} \mathcal{N}(u_{n}) \to +\infty.
\]
\end{theorem}

\section{Proof of Theorem \ref{arith}}
We now prove Theorem \ref{arith} using Theorem \ref{thm1}. Let $\Gamma$ be a arithmetic triangle group and let $\mathbb{X}=\Gamma \backslash \mathbb{H}$. Let $\{\phi_j\}_j$ be the complete sequence of Hecke--Maass eigenforms on $\mathbb{X}$, i.e., it is a joint eigenfunction of $-\Delta_g$ and Hecke operators $\{T_n\}_{n\geq 1}$. It is shown in \cite{GRS2} that there exists an orientation-reversing isometric involution $\tau: \mathbb{X} \to \mathbb{X}$ such that $\mathrm{Fix} (\tau)$ is separating and that $\tau$ commutes with all $T_n$. From the multiplicity one theorem for Hecke eigenforms \cite{MR0268123}, the sequence of Hecke eigenvalues $\{\lambda_\phi(n)\}_{n \geq 1}$ of $T_n$ (i.e., $T_n\phi=\lambda_\phi(n) \phi$) determines $\phi$ uniquely. Hence any Hecke--Maass eigenform $\phi_j$ on $\mathbb{X}$ is an eigenfunction of $\tau$ so that we have either $\tau \phi_j =\phi_j$ or $\tau \phi_j=-\phi_j$ for all $j$. Now from the arithmetic Quantum Unique Ergodicity theorem by Lindenstrauss\cite{lin06}, QUE holds for $\{\phi_j\}_j$---hence we conclude that $\lim_{j \to +\infty} \mathcal{N}(\phi_j)=+\infty$ by Theorem \ref{thm1}.

\appendix
\section{Semiclassical analysis on manifolds}
In this section we review pseudo-differential operators and quantization of symbols on $T^*M$ from \cite{zss}. We refer the reader to \cite{zss} for details.
\subsection{The standard quantization of a symbol}
Recall that a class symbols for which we have invariance under coordinate changes is given by
\[
%S^{k,m}=\left\{a\in C^\infty(\mathbb{R}^{2n})~:~ |\partial_x^\alpha\partial_\xi^\beta a| \leq C_{\alpha\beta} h^{-k}|\xi|^{m-|\beta|}\ \forall\alpha,\beta\right\}
S^m = \left\{ a\in C^\infty(\mathbb{R}^{2n})~:~ |\partial_x^\alpha\partial_\xi^\beta a|\leq C_{\alpha\beta} |\xi|^{m-|\beta|}\ \forall\alpha,\beta\right\}.
\]
%The index $k$ records how singular the symbol $a$ is as $h \to 0$, and $m$ controls the growth rate as $|\xi| \to \infty$. Let $\mathscr{S}$ be the Schwartz space given by
The index $m$, called the \emph{order}, controls the growth rate as $|\xi|\to \infty$. Let $\mathscr{S}$ be the Schwartz space given by
\[
\mathscr{S}=\mathscr{S}(\mathbb{R}^n)= \{\phi\in C^\infty(\mathbb{R}^n)~:~ \sup_{\mathbb{R}^n}|x^\alpha \partial^\beta\phi| <\infty \forall \alpha,\beta  \}.
\]
\begin{definition}[\S4.1 \cite{zss}]
We define the \textit{standard quantization}
\[
a(x,hD)u(x) := \frac{1}{(2\pi h)^n} \int_{\mathbb{R}^n}\int_{\mathbb{R}^n} e^{\frac{i}{h}\langle x-y,\xi \rangle}a(x,\xi) u(y) dyd\xi
\]
for $a \in S^{k,m}$ and for $u \in \mathscr{S}$. Here, we view $D=-i\partial$.
\end{definition}
For example if $a(x,\xi)=\sum_{|\alpha|<N}a_\alpha(x)\xi^\alpha$, then
\[
a(x,hD)=\sum_{|\alpha|\leq N} a_\alpha(x) (hD)^\alpha u=\sum_{|\alpha|\leq N}a_\alpha(x)(-ih\partial)^\alpha u.
\]

\subsection{Asymptotic series}

\begin{definition}[\S4.4.2 \cite{zss}]
Let $a_j\in S^{m}$, for $j\in\mathbb{Z}_{\geq 0}$. We say $a\in S^m$ is \emph{asymptotic} to $\sum_{j=0}^\infty h^ja_j$, written as $a\sim\sum_{j=0}^\infty h^ja_j$ in $S^m$, if
\[
\left|\partial_x^\alpha\partial_\xi^\beta\left(a-\sum_{j=0}^{N-1}h^ja_j\right)\right|\leq C_{\alpha\beta N}h^N|\xi|^{m-|\beta|}
\]
for all multiindices $\alpha,\beta$ and $N\in\mathbb{Z}_{>0}$. (Note that we are \emph{not} expecting for the series $\sum_{j=0}^\infty h^ja_j$ to converge in any sense.) If $a\sim\sum_{j=0}^\infty h^ja_j$ in $S^m$, then we say $a_0$ to be a \emph{principal symbol} of $a$. (Note that principal symbols may vary up to $O(h)$.)
\end{definition}

Any symbol has an asymptotic expression, namely $a\sim a+\sum_{j=1}^\infty h^j\cdot 0$. Conversely, given an asymptotic expression, we can always find a symbol that has that asymptotic expression.

\begin{theorem}[Borel's theorem, \cite{zss} Theorem 4.15] \label{thm:Borel}
Assume $a_j\in S^{m}$ for $j\in\mathbb{Z}_{\geq 0}$. Then there exists a symbol $a\in S^{m}$ such that $a\sim\sum_{j=0}^\infty h^ja_j$ in $S^{m}$. If also $\hat{a}\sim\sum_{j=0}^\infty h^ja_j$, then $a-\hat{a}=O_{S^{m},N}(h^N)$ for any $N\in\mathbb{Z}_{>0}$.
\end{theorem}

\subsection{Pseudodifferential operators on manifolds}

\begin{definition}[\S14.2.2 of \cite{zss}]
A linear operator
\[
A:C^\infty(M) \to C^\infty (M)
\]
is called a \textit{pseudo-differential operator} of order $m$, if $m\in\mathbb{Z}$ is such that for each coordinate patch $(U,p)$, there exists a symbol $a \in S^{m}(\mathbb{R}^{2n})$ such that for any $\phi,\psi\in C_0^\infty (U)$ and for each $u \in C^\infty(M)$,
\[
\phi A(\psi u)=\phi p^* a(x,hD)(p^{-1})^* (\psi u),
\]
and if for any $\phi_j \in C_0^\infty (M)$, $j=1,2$ satisfying $\mathrm{supp}(\phi_1) \cap \mathrm{supp}(\phi_2) = \emptyset$, we have
\[
\phi_1 A \phi_2 = O_N(h^N):H^{-k}(U_2) \to H^k (U_1),
\]
for all $k\in\mathbb{Z}_{\geq 0}$ and $N\in\mathbb{Z}_{>0}$, where $U_j\subset\subset M$ are open neighbourhoods of $\mathrm{supp}(\phi_j)$, $j=1,2$, and $H^{\pm k}(U_j)$ are the Sobolev spaces
\[
H^k(U_1)=\left\{u\in L^2(U_1){\Bigg |}\forall 1\leq l\leq k\forall V_1,\cdots,V_l\in\Gamma(U_1,TM);V_1\cdots V_l u\in L^2(U_1)\right\},
\]
\[
H^{-k}(U_2)=\mathrm{span}_\mathbb{C}\left\{V_1\cdots V_l f~{\Bigg |}~0\leq l\leq k,V_1,\cdots,V_l\in \Gamma(U_2,TM),f\in L^2(U_2)\right\}.
\]
(Here, $\Gamma(U,TM)$ is the collection of all smooth vector fields on $U$.) %or $\Gamma(-,TM)$ being the sheaf of smooth sections of the vector bundle $TM$
The collection of all pseudo-differential operators on $M$ of order $m$ is denoted $\Psi^{m}(M)$.
\end{definition}

\begin{definition}[\S14.2.3 of \cite{zss}]
We say $a\in S^m(T^\ast M)$, $a$ being a \emph{symbol} on $T^\ast M$ of order $m$, if $a\in C^\infty(T^\ast M)$ and for each coordinate patch $(U,p)$, under the pullback through the identification and inclusion $p[U]\times\mathbb{R}^n\xrightarrow{\sim}T^\ast U\subset T^\ast M$, $a$ becomes a symbol in $S^m(p[U]\times\mathbb{R}^n)$.
\end{definition}

The following theorem describes how quantizations on manifolds works, i.e., how we relate symbols and pseudo-differential operators. This also suggests a way to detect the \emph{principal symbol}.

\begin{theorem}[Theorem 14.1 of \cite{zss}]
There exist linear maps $\sigma\colon\Psi^m(M)\rightarrow S^m(T^\ast M)/hS^{m-1}(T^\ast M)$ and $\mathrm{Op}\colon S^m(T^\ast M)\rightarrow\Psi^m(M)$ such that $\sigma(A_1A_2)=\sigma(A_1)\sigma(A_2)$ and $\sigma(\mathrm{Op}(a))=[a]\in S^m(T^\ast M)/hS^{m-1}(T^\ast M)$. Here, $[a]$ denotes the equivalence class represented by $a\in S^m(T^\ast M)$ in the quotient ring
\[
S^m(T^\ast M)/hS^{m-1}(T^\ast M).
\]
\end{theorem}

Recall that, locally, a symbol $a$ has a principal symbol that may differ up to $O(h)$. This is captured by the fact that $\sigma(\mathrm{Op}(a))=[a]$. Likewise, it is harmless to define, given a pseudo-differential operator $A\in\Psi^m(M)$, its \emph{principal symbol} as a representative of $\sigma(A)$. Since any two representatives of $[a]\in S^m(T^\ast M)/hS^{m-1}(T^\ast M)$ have $O(h)$ difference, this implies that the principal symbol is independent to the quantization (that is, any two choice shows no difference as $h\rightarrow 0$).

\subsection{QUE for a symbol of finite order}\label{finite}
\begin{lemma}[Theorem 6.4 of \cite{zss}]\label{cutoff}
Let $(M,g)$ be a smooth compact Riemannian manifold and suppose that $u \in L^2(M,g)$ satisfies
\[
-h^2\Delta_g u = E(h)u.
\]
Assume as well that $a \in S^m$ is a symbol satisfying, for some $E>0$,
\[
\mathrm{supp}(a)\cap\{(x,\xi) \in T^*M ~:~ |\xi|_x^2 = E\}=\emptyset.
\]
Then if $|E(h)-E| < \delta$ for some sufficiently small constant $\delta>0$, we have the estimate
\[
\|a(x,hD)u\|_{L^2}=O_N(h^{N})\|u\|_{L^2}
\]
for any $N>0$.
\end{lemma}
\begin{proof}
Let $\{(U_i,p_i)\}_{i\in I}$ be a finite set of coordinate patch such that $\bigcup_{i \in I}U_i = M$, and $p_i[U_i]\subset\subset\mathbb{R}^n$. Let $\{\psi_j\}_{j \in J}$ be a partition of unity such that for any $j \in J$, $\mathrm{supp}(\psi_j) \subset\subset U_{i}$ for some $i\in I$. Then it is sufficient to prove the Theorem for each $\psi_j (x) a(x,\xi)$.

Fix $j$ and $i$ such that $\mathrm{supp}(\psi_j) \subset U_i$. Let $\tilde{a}$ be a symbol in $S^m(p_i[U_i]\times \mathbb{R}^n)$ which is a pullback of $\psi_j(x)a(x,\xi)$ under the identification $p_i[U_i]\times\mathbb{R}^n\xrightarrow{\sim}T^\ast U_i$.

Our goal is to estimate the norm of $\psi_j(x)a(x,hD)u\in L^2(M)$. However, the function $\psi_j(x)a(x,hD)u$ is supported on $\mathrm{supp}(\psi_j)$, so it is harmless to view
\[
\psi_j(x)a(x,hD)u=\tilde{a}(x,hD)u\in L^2(\mathbb{R}^n).
\]
Now, set a smooth $\phi\colon \mathbb{R}^n\rightarrow[0,1]$, whose support $\subset\subset p_i[U_i]$, and $\phi\equiv 1$ on some neighborhood of $\mathrm{supp}(\psi_j)$. In particular, $\mathrm{supp}(1-\phi)\cap\mathrm{supp}(\psi_j)=\emptyset$. Then by the fact that $a(x,hD)$ is a pseudo-differential operator, we have
\begin{eqnarray*}
\psi_j(x)a(x,hD)u &=& \psi_j(x)a(x,hD)(\phi u)+\psi_j(x)a(x,hD)((1-\phi)u) \\
&=& \psi_j(x)a(x,hD)(\phi u)+O_N(h^N)\|u\|_{L^2(M)}.
\end{eqnarray*}
We thus claim that $\tilde{a}(x,hD)(\phi u)=O_N(h^N\|u\|_{L^2})$.

Identify $g$ with the pullback tensor of $g|_{U_i}$ through the identification $p_i[U_i]\xrightarrow{\sim}U_i$. Then, we can extend $g$ to be a metric tensor on all $\mathbb{R}^n$, such that $\Lambda I\geq g\geq\Lambda^{-1}I$ for some $\Lambda>1$. Write $g=(g_{k\ell})_{k,\ell}$, so that $g^{-1}=(g^{k\ell})_{k,\ell}$.

Set $\chi\colon\mathbb{R}^{n}\times\mathbb{R}^{n}\rightarrow[0,1]$ to be a smooth, compactly supported function with $\chi\equiv 1$ on $\{(x,\xi)\in p_i[U_i]\times\mathbb{R}^n~|~E-\delta\leq|\xi|_x^2\leq E+\delta\}$ and $\chi\equiv 0$ on $\mathrm{supp}(\tilde{a})$. Set a symbol $b\in S^2$ by
\[
b=\sum_{k,\ell=1}^n\left(\frac{1}{\sqrt{\det g}}\frac{h}{i}\frac{\partial(\sqrt{\det g}g^{k\ell})}{\partial x_k}\xi_\ell+g^{k\ell}\xi_k\xi_\ell\right)-E(h)\phi+(1-\phi)+i\chi,
\]
where $\phi$ is the cutoff function introduced above; this is added to make sure that $|b(x,\xi)|\geq\gamma(1+|\xi|^2)$ for some $\gamma>0$. In particular, $1/b\in S^0$ holds.

Now recursively define symbols $c_0,c_1,\cdots$ by, $c_0=1/b$ and
\[
c_j=-\frac{1}{b}\sum_{1\leq|\alpha|\leq j}\frac{1}{\alpha!}\partial_\xi^\alpha c_{j-|\alpha|}D_x^\alpha b.
\]
Here, one verifies that $c_j\in S^0$ for all $j$. Now define a symbol $c\in S^{0}$ with asymptotic $c\sim\sum_{j=0}^\infty h^jc_j$, appealing to Borel's theorem (Theorem \ref{thm:Borel} above). Now by \cite[Theorem 9.5]{zss}, if we set
\[
s(x,\xi)=e^{ih\langle D_\xi, D_y\rangle}c(x,\xi)b(y,\eta)|_{\substack{ y=x \\ \eta=\xi}}
\]
then we have $c(x,hD)b(x,hD)=s(x,hD)$. This $s$ satisfies, for any $N\geq 3$, by the same theorem,
\[
s(x,\xi)=\sum_{|\alpha|\leq N}\frac{1}{\alpha!}\partial_\xi^\alpha\left(\sum_{j=0}^\infty h^jc_j\right)(hD_x)^\alpha b +O_N(h^{N+1})=1+O_N(h^N),
\]
by the virtue of the definitions of $c_j$'s. By the operator size estimate \cite[Theorem 4.23]{zss} we have $s(x,hD)=I+O_N(h^N)$ for every $N\geq 3$, so we have
\[
\tilde{a}(x,hD)c(x,hD)b(x,hD)=\tilde{a}(x,hD)+O_{N}(h^N)
\]
consequently. Furthermore, since $\mathrm{supp}(\tilde{a})\cap\mathrm{supp}(\chi)=\emptyset$, we have  \cite[Theorem 4.25]{zss}, or by direct computation using \cite[Theorem 9.5]{zss})
\[
\tilde{a}(x,hD)c(x,hD)\chi(x,hD)=O_{N}(h^N).
\]
Now we estimate the $L^2$ norm of $\tilde{a}(x,hD)(\phi u)$. First, on $\mathrm{supp}(\psi_j)$, as $b(x,hD)=-h^2\Delta_g-E(h)+i\chi(x,hD)$ on there,
\begin{align*}
&\tilde{a}(x,hD)(\phi u) \\
=& \tilde{a}(x,hD)c(x,hD)b(x,hD)(\phi u)+O_N(h^N\|u\|_{L^2}) \\
=& \tilde{a}(x,hD)c(x,hD)(-h^2\Delta_g-E(h)+i\chi(x,hD))(\phi u)+O_N(h^N\|u\|_{L^2}) \\
=& \tilde{a}(x,hD)c(x,hD)(-h^2\Delta_g-E(h))(\phi u)+O_N(h^N\|u\|_{L^2}).
\end{align*}
Here, note that $\tilde{a}(x,hD)c(x,hD)=s_1(x,hD)$, for some symbol $s_1$ supported on $\mathrm{supp}(\psi_j)\times\mathbb{R}^n$ (by  \cite[Theorem 9.5]{zss}). However, as $(-h^2\Delta_g-E(h))(\phi u)=0$ on $\{\phi=1\}\supset\supset\mathrm{supp}(\psi_j)$, this support disagreement implies that
\[
\tilde{a}(x,hD)c(x,hD)(-h^2\Delta_g-E(h))(\phi u)=O_N(h^N)\|u\|_{L^2}.
\]
Together with $\tilde{a}(x,hD)(\phi u)\equiv0$ on $\mathbb{R}^n\setminus\mathrm{supp}(\psi_j)$, we conclude that
\[
\left\|\tilde{a}(x,hD)(\phi u)\right\|_{L^2}=O_N(h^N)\|u\|_{L^2}.\qedhere
\]
\end{proof}
We now consider the eigenvalue problem $-\Delta_g u = \lambda^2 u$. In this context, we define the action of $a$ on $u$ as follows:
\[
\mathrm{Op}(a)u = a\left(x,\frac{1}{\lambda}D\right)u.
\]
\begin{lemma}\label{cutoff2}
Let $(M,g)$ be a smooth compact Riemannian manifold and suppose that $u \in L^2(M,g)$ satisfies
\[
-\Delta_g u = \lambda^2  u.
\]
Assume that $a \in C_0^\infty (T^*M)$ satisfies
\begin{equation}\label{cond}
\mathrm{supp} (a(x,\xi)) \subset \{(x,\xi) \in T^*M~:~ 1-\delta <|\xi|_x^2<1+\delta \}
\end{equation}
for some $\delta>0$ and
\begin{equation}\label{cond2}
a(x,\xi)=0~~ \mathrm{if}~ |\xi|_x=1.
\end{equation}
Then
\[
\|\mathrm{Op}(a) u\|_{L^2(M)} = O(\lambda^{-1}).
\]
\end{lemma}
\begin{proof}
We use the same notation as in the proof of Lemma \ref{cutoff}, where the only difference is that the symbol $a$ now satisfies \eqref{cond} and \eqref{cond2}. Hence, it is sufficient to prove that
\[
\|\mathrm{Op}(\tilde{a})(\phi u)\|_{L^2} = O(\lambda^{-1}\|\phi u\|_{L^2}),
\]
where $\tilde{a} (x,\xi) = 0$ if $|\xi|_x=1$. Let $b \in C_0^\infty (\mathbb{R}^{2n})$ be given by
\[
(|\xi|_x^2-1)b(x,\xi)=\tilde{a}(x,\xi).
\]
Note that $b$ has support in $\mathrm{supp}(\psi_j)\times\mathbb{R}^n$, and that $(-\frac{1}{\lambda^2}\Delta_g-1)(\phi u)=0$ on $\mathrm{supp}(\psi_j)\subset\subset\{x~:~\phi(x)=1\}$. Therefore by \cite[Theorem 4.25]{zss}, we have
\[
\|\mathrm{Op}(b)(-\frac{1}{\lambda^2}\Delta_g-1)(\phi u)\|_{L^2}=O_N(\lambda^{-N}\|\phi u\|_{L^2}),
\]
for any $N \in \mathbb{N}$. Now observe that
\begin{align*}
\|\mathrm{Op}(b)(-\frac{1}{\lambda^2}\Delta_g -1)(\phi u)\|_{L^2} = \|\mathrm{Op}(b) \mathrm{Op}(|\xi|_x^2-1)(\phi u)\|_{L^2} + O(\lambda^{-1}\|\phi u\|_{L^2}),
\end{align*}
and from  \cite[Theorem 4.14]{zss} that
\begin{align*}
 \|\mathrm{Op}(b) \mathrm{Op}(|\xi|_x^2-1)(\phi u)\|_{L^2} &=\| \mathrm{Op}(b(|\xi|_x^2-1))(\phi u)\|_{L^2} + O(\lambda^{-1}\|\phi u\|_{L^2})\\
 &=\| \mathrm{Op}(\tilde{a})(\phi u)\|_{L^2} + O(\lambda^{-1}\|\phi u\|_{L^2}).
\end{align*}
We therefore conclude that
\[
\| \mathrm{Op}(\tilde{a})(\phi u)\|_{L^2}=O(\lambda^{-1}\|\phi u\|_{L^2}). \qedhere
\]
\end{proof}
A homogeneous symbol of degree $k$ is defined to be a function
\[
a(x,\xi) \in C^\infty (T^*M-\{0\})
\]
such that $a(x,t\xi) = t^k a(x,\xi)$ for any $t>0$.
\begin{theorem}[Quantum Ergodicity Theorem \cite{snirel,cd22,zeld1}]\label{thm:qe02}
Assume ergodic geodesic flow on $M$. For any given orthonormal eigenbasis $\{u_n\}$ of $L^2(M)$, there exists a density $1$ subsequence $\{u_{n_k}\} \subset \{u_n\}$ such that for any homogeneous symbol $a$ of degree $0$, we have
\begin{equation}\label{homo}
\lim_{k \to \infty }\langle \mathrm{Op}(a)u_{n_k} , u_{n_k}\rangle = \int_{S^*M} a(x,\xi) dxd\xi.
\end{equation}
\end{theorem}
We now prove that if \eqref{homo} is true for all degree $0$ symbol, then it is also true for any symbol of finite order.
\begin{theorem}
Assume that $\{u_{n_k}\}$ is a sequence of eigenfunctions such that \eqref{homo} is true for all homogeneous degree $0$ symbol. Then for any $a \in S^m(M)$, we have
\begin{equation*}
\lim_{k \to \infty }\langle \mathrm{Op}(a)u_{n_k} , u_{n_k}\rangle = \int_{S^*M} a(x,\xi) dxd\xi.
\end{equation*}
\end{theorem}
\begin{proof}
Fix $a \in S^m(M)$, and let $\tilde{a}(x,\xi):= a(x,\xi/|\xi|_x)$. Let $\chi \in C_0^\infty (\mathbb{R})$ be a function such that $\mathrm{supp}(\chi) \subset (1-2\delta, 1+2\delta)$ and that $\chi(y)=1$ if $y \in (1-\delta, 1+\delta)$ for some small fixed $\delta$. Then by Lemma \ref{cutoff}, we have
\begin{equation}\label{esttt1}
\langle \mathrm{Op}((a-\tilde{a})(1-\chi(|\xi|_x^2)))u_{n_k} , u_{n_k}\rangle = O_N(\lambda_{n_k}^{-N}).
\end{equation}
Observing that $a(x,\xi)-\tilde{a}(x,\xi)=0$ if $|\xi|_x=1$, we apply Lemma \ref{cutoff2} to have
\begin{equation}\label{esttt2}
\langle \mathrm{Op}((a-\tilde{a})\chi(|\xi|_x^2))u_{n_k} , u_{n_k}\rangle = O(\lambda_{n_k}^{-1}).
\end{equation}
Combining \eqref{esttt1} and \eqref{esttt2}, we conclude that
\[
\langle \mathrm{Op}(a)u_{n_k} , u_{n_k}\rangle=\langle \mathrm{Op}(\tilde{a})u_{n_k} , u_{n_k}\rangle+O(\lambda_{n_k}^{-1}),
\]
hence
\begin{align*}
\lim_{k \to \infty}\langle \mathrm{Op}(a)u_{n_k} , u_{n_k}\rangle&=\lim_{k \to \infty} \langle \mathrm{Op}(\tilde{a})u_{n_k} , u_{n_k}\rangle\\
&=\int_{S^*M} \tilde{a}(x,\xi) dxd\xi\\
&=\int_{S^*M} a(x,\xi) dxd\xi. \qedhere
\end{align*}
\end{proof}

\bibliographystyle{alpha}
\bibliography{bibfile}
\end{document}